\documentclass[twoside,a4paper,reqno,11pt]{amsart} 
\usepackage{amsfonts, amsbsy, amsmath, amssymb, latexsym}
\usepackage{mathrsfs}

\usepackage[top=30mm,right=30mm,bottom=30mm,left=30mm]{geometry}
\usepackage{stmaryrd}
\usepackage{bm}

\renewcommand{\a}{\alpha}

\newcommand{\e}{\epsilon}

 \renewcommand{\to}{\rightarrow}
 \newcommand{\s}{\sigma}

\newcommand{\leqs}{\leqslant}
\newcommand{\geqs}{\geqslant}

 \newcommand{\vs}{\vspace{3mm}}

\newtheorem{theorem}{Theorem}

\newtheorem{thm}{Theorem}[section]
\newtheorem{prop}[thm]{Proposition}
\newtheorem{lem}[thm]{Lemma}
\newtheorem{cor}[thm]{Corollary}

\theoremstyle{definition}

\newtheorem{remk}{Remark}

\begin{document}

 \author{Timothy C. Burness}
 \address{School of Mathematics, University of Bristol, Bristol BS8 1TW, UK}
 \email{t.burness@bristol.ac.uk}
 
\author{\'{A}kos Seress}
\address{Department of Mathematics, The Ohio State University, 231 W 18th Avenue, Columbus OH 43210, USA}
\email{akos@math.ohio-state.edu}

\title
[On Pyber's base size conjecture]{On Pyber's base size conjecture}

\thanks{To appear in \emph{Transactions of the American Mathematical Society}. The first author was supported by EPSRC grant EP/I019545/1, and he thanks the Department of Mathematics at The Ohio State University for its generous hospitality. Both authors thank Professor C.E. Praeger for helpful comments.}



\begin{abstract}
Let $G$ be a permutation group on a finite set $\Omega$.  A subset of $\Omega$ is a base for $G$ if its pointwise stabilizer in $G$ is trivial. The base size of $G$, denoted $b(G)$, is the smallest size of a base. A well known conjecture of Pyber from the early 1990s asserts that there exists an absolute constant $c$ such that $b(G) \leqs c\log |G| / \log n$ for any primitive permutation group $G$ of degree $n$. 
Some special cases have been verified in recent years, including the almost simple and diagonal cases.
In this paper, we prove Pyber's conjecture for all non-affine primitive groups.
\end{abstract}

\date{\today}
\maketitle
 
\section{Introduction}\label{s:intro}

Let $G$ be a permutation group on a set $\Omega$ of size $n$.  A subset $B$ of $\Omega$ is a base for $G$ if the pointwise stabilizer of $B$ in $G$ is trivial. The \emph{base size} of $G$, denoted $b(G)$, is the smallest size of a base for $G$. Determining base sizes is a fundamental problem in permutation group theory, with a long history stretching back to the nineteenth century. More recently, bases have played an important role in the computational study of finite permutation groups (see \cite[Chapter 4]{Seress_book} for further details).

Clearly, the elements of $G$ are uniquely determined by their effect on a base. In particular, if $B$ is a base for $G$ then $|G| \leqs n^{|B|}$ and thus
$$b(G) \geqs \frac{\log |G|}{\log n}.$$ 
A well known conjecture of L. Pyber \cite[p.207]{Pyber} asserts that there is an absolute constant $c$ such that 
$$b(G) \leqs c \frac{\log |G|}{\log n}$$
for any primitive group $G$ of degree $n$. It is easy to see that the primitivity condition is necessary. For instance, $G = Z_2 \wr Z_k$ is a transitive, imprimitive group of degree $2k$, and $b(G) = k = \log_2 |G| - \log_2 k$. 

In recent years, several special cases of Pyber's conjecture have been verified by various authors (see Section \ref{s:prel} for more details):
\begin{itemize}\addtolength{\itemsep}{0.5\baselineskip}
\item[(i)] Almost simple groups: Liebeck and Shalev \cite{LS0},
  Burness et al. \cite{Bur,BGS,BLS,BOW} (non-standard actions); Benbenishty \cite{Benbenishty} (standard actions).
\item[(ii)] Diagonal groups: Fawcett \cite{Fawcett}.
\item[(iii)]  Affine groups (some special cases): Liebeck and Shalev \cite{LS} (primitive case); Seress \cite{Seress} (solvable case); Gluck and Magaard \cite{GM} (coprime case).
\end{itemize}

A related (but somewhat weaker) conjecture of Babai on base sizes for primitive groups was proved in \cite{GSS}. For a positive integer $d$, let $\Gamma_d$ be the family of finite groups $G$ with the property that $G$ has no alternating composition factors of degree greater than $d$, and no classical composition factors of rank greater than $d$ (there are no conditions on the cyclic, exceptional and sporadic composition factors of $G$). Then \cite[Theorem 1.2]{GSS} implies that Pyber's conjecture holds for all $G \in \Gamma_d$ (with the constant $c$ depending on $d$).

Our main result establishes Pyber's conjecture for all non-affine primitive groups. In particular, this is the first paper to consider the conjecture for product-type and twisted wreath product primitive groups.

\begin{theorem}\label{tt:main}
There exists an absolute constant $c$ such that 
$$b(G) \leqs c \frac{\log |G|}{\log n}$$
for any non-affine primitive group $G$ of degree $n$.
\end{theorem}

\begin{remk}\label{r:1}
By combining Theorem \ref{tt:main} with the main results in \cite{LS, Seress, GM} on affine primitive groups, it follows that in order to complete the proof of Pyber's conjecture one may assume that  $G = V \rtimes G_0$ is an affine group, where $|V|=p^d$ ($p$ a prime), $G_0 \leqs {\rm GL}(V)$ is irreducible and the following three conditions hold:
\begin{itemize}\addtolength{\itemsep}{0.5\baselineskip}
\item[(i)] $G_0$ is nonsolvable;  
\item[(ii)]  $p$ divides $|G_0|$; and 
\item[(iii)] $G_0$ acts imprimitively on $V$ (that is, $G_0$ stabilizes a nontrivial direct sum decomposition of $V$).
\end{itemize}
\end{remk}

The main focus of this paper concerns the primitive groups of product-type. In Section \ref{ss:pybpt} we will deduce Pyber's conjecture for product-type groups by applying the following theorem, which may be of independent interest.

\vs

\noindent \textbf{Theorem \ref{t:main}.} \emph{Let $P$ be a transitive permutation group of degree $k$. Then there exist
$$O\left(1+\frac{\log_2 |P|}{k}\right)$$
$2$-part partitions of $\{1, \ldots, k\}$ with the property that the intersection of the stabilizers of these partitions in $P$ is trivial.}

\vs

This is relatively straightforward to establish when $P$ is primitive (see Section \ref{ss:prim}), but more effort is required in the imprimitive case. The latter situation is handled in Section \ref{ss:imprim}, and the highly combinatorial proof uses (and extends) some of the ideas in the proof of \cite[Theorem 4.1]{GSS}, which plays a key role in the proof of Babai's base size conjecture in \cite{GSS}. Finally, in Section \ref{s:tw} we establish Pyber's conjecture for twisted wreath products as an easy corollary of our work on product-type groups. 

\vs

\noindent \textbf{Notation.}
Our notation is fairly standard. For a positive integer $n$ we set
$[n]=\{1, \ldots, n\}$. All logarithms in this paper are with respect
to the base $2$. If $X$ is a set then we write $(X_1, \ldots, X_n)$
to denote an ordered partition of $X$ into subsets $X_i$. For $G \leqs {\rm Sym}(\Omega)$ and $\Gamma \subseteq \Omega$, we denote by $G_{\Gamma}$ the setwise stabilizer of $\Gamma$ in $G$. Also, if $G$ acts on a set $\Delta$ then $G^{\Delta}$ denotes this action. Finally, we write ${\rm Soc}(G)$ for the socle of a finite group $G$, that is, the subgroup of $G$ generated by its minimal normal subgroups. Some additional notation will be introduced as and when necessary.

\section{Preliminaries}\label{s:prel} 

\subsection{The O'Nan-Scott theorem}\label{ss:onan}

Let $G$ be a primitive permutation group of degree $n$. The various possibilities for $G$ are described by the O'Nan-Scott theorem, which classifies the finite primitive permutation groups according to their socle  and the action of a point stabilizer. Here we follow the version presented by Liebeck, Praeger and Saxl \cite{LPS}, with Table \ref{t:onan} providing a rough description of the families of primitive groups that arise. In the table, $V$ is a vector space over a prime field $\mathbb{F}_{p}$, $T$ denotes a nonabelian simple group, and $T^k$ is the direct product of $k$ copies of $T$. We refer the reader to \cite{LPS} for further details (we will say more about product-type groups in Section \ref{s:pa}).

\renewcommand{\arraystretch}{1.1}
\begin{table}
$$\begin{array}{ll} \hline
\mbox{Type} & \mbox{Description} \\ \hline
\mbox{I} & \mbox{Affine: $G = V \rtimes G_0 \leqs {\rm AGL}(V)$, $G_0 \leqs{\rm GL}(V)$ irreducible} \\
\mbox{II} & \mbox{Almost simple: $T \leqs G \leqs {\rm Aut}(T)$} \\
\mbox{III(a)(i)} & \mbox{Diagonal-type: $T^k \leqs G \leqs T^k.({\rm Out}(T) \times P)$, $P \leqs S_k$ primitive} \\
\mbox{III(a)(ii)} & \mbox{Diagonal-type: $T^2 \leqs G \leqs T^2.{\rm Out}(T)$} \\
\mbox{III(b)(i)} & \mbox{Product-type: $G \leqs H \wr P$, $H$ primitive of type II, $P \leqs S_k$ transitive} \\
\mbox{III(b)(ii)} & \mbox{Product-type: $G \leqs H \wr P$, $H$ primitive of type III(a), $P \leqs S_k$ transitive} \\
\mbox{III(c)} & \mbox{Twisted wreath product} \\ \hline
\end{array}$$
\caption{The primitive permutation groups}
\label{t:onan}
\end{table}
\renewcommand{\arraystretch}{1}

As noted in the Introduction, Pyber's conjecture has been verified in the almost simple and diagonal-type cases referred to in Table \ref{t:onan}. In addition, several special cases involving affine groups have also been handled. Below we briefly summarize the progress to date on Pyber's conjecture, starting with the almost simple groups.

\subsection{Almost simple groups}\label{ss:as}

Let $G \leqs {\rm Sym}(\Omega)$ be an almost simple primitive group with socle $T$. In the study of such groups, it is natural to make a distinction between the so-called \emph{standard} and \emph{non-standard} groups. Roughly speaking, we say that $G$ is standard if $T=A_n$ is an alternating group and $\Omega$ is an orbit of subsets or partitions of $[n]$, or $T=Cl(V)$ is a classical group and $\Omega$ is an orbit of subspaces of the natural $T$-module $V$ (see \cite[Definition 1.1]{Bur} for the precise definition). 

\begin{thm}\label{t:as}
Pyber's conjecture holds if $G$ is almost simple. 
\end{thm}

\begin{proof}
For standard groups, a strong form of the conjecture (with an explicit constant $c < 15$) is established in the unpublished PhD thesis of Benbenishty \cite{Benbenishty} (see \cite{BCN, Halasi, James} for additional results on bases for standard actions of alternating and symmetric groups).

In general, the base size of a standard group can be arbitrarily
large, which is in stark contrast to the situation for non-standard
groups. Indeed, a theorem of Liebeck and Shalev \cite[Theorem
1.3]{LS0} states that there is an absolute constant $c$ such that
$b(G) \leqs c$ for any non-standard group $G$ (and moreover, the
probability that a random $c$-tuple of points in $\Omega$ forms a base for
$G$ tends to $1$ as $|G|$ tends to infinity). Later, in a series of
papers \cite{Bur,BGS,BLS,BOW}, Burness et al. showed that $b(G) \leqs 7$, with equality if and only if $G$ is the largest Mathieu group ${\rm M}_{24}$ in its $5$-transitive action of degree $24$.
\end{proof}

\subsection{Diagonal groups}\label{ss:d}

Let $G$ be a primitive diagonal-type group of degree $n$ with socle $T^k$, where $T$ is a nonabelian simple group and $k \geqs 2$. Then $n=|T|^{k-1}$ and $G$ is a (not necessarily split) extension of $T^k$ by a subgroup of ${\rm Out}(T) \times S_k$. Let $P \leqs S_k$ be the permutation group induced from the conjugation action of $G$ on the $k$ factors of $T^k$. The primitivity of $G$ implies that either $P$ is primitive (as in Case III(a)(i) in Table \ref{t:onan}), or $k=2$ and $P=1$ (as in Case III(a)(ii)). As explained in \cite{Fawcett}, the induced group $P$ plays a large part in determining $b(G)$.

\begin{thm}\label{t:d}
Pyber's conjecture holds if $G$ is a diagonal-type group. 
\end{thm}

\begin{proof}
In \cite{Fawcett}, Fawcett establishes a strong form of Pyber's conjecture for diagonal groups. Indeed,  \cite[Theorem 1.3]{Fawcett} states that
$$b(G) \leqs \left\lceil \frac{\log |G|}{\log n}\right\rceil +2$$
(see also \cite[Remark 4.3]{GSS}). More precisely, the following results are proved (we refer the reader to \cite{Fawcett} for more detailed results in (ii) and (iii) below): 
\begin{itemize}\addtolength{\itemsep}{0.5\baselineskip}
\item[(i)] If $P \neq A_k,S_k$ then $b(G)=2$;
\item[(ii)] If $k=2$ then $b(G) \in \{3,4\}$; 
\item[(iii)] If $k \geqs 3$ and $A_k \leqs P$ then 
$$b(G) =  \left\lceil \frac{\log k}{\log |T|} \right\rceil+\e$$
with $\e \in \{1,2\}$. 
\end{itemize}
In addition, see \cite[Section 4]{Fawcett} for some interesting probabilistic results concerning bases for diagonal-type groups.
\end{proof}

\subsection{Affine groups}\label{ss:a}

Let $G$ be a primitive affine group of degree $n$. Then $n=p^d$ for a prime $p$, and we have 
$$G = V \rtimes G_0 \leqs V \rtimes {\rm GL}(V) = {\rm AGL}(V)$$
where $V$ is a $d$-dimensional vector space over $\mathbb{F}_{p}$, and $G_0 \leqs {\rm GL}(V)$ is irreducible. In this situation, several special cases of Pyber's conjecture have been verified by various authors, but the general case remains open (see Remark \ref{r:1}).

\begin{thm}\label{t:affine}
Let $G=V \rtimes G_0$ be a primitive affine group, with $|V|=p^d$ for a prime $p$. Then the conclusion to Pyber's conjecture holds if one of the following holds:
\begin{itemize}\addtolength{\itemsep}{0.5\baselineskip}
\item[{\rm (i)}] $G_0$ is solvable;
\item[{\rm (ii)}] $|G_0|$ is indivisible by $p$; 
\item[{\rm (iii)}] $G_0$ acts primitively on $V$.
\end{itemize}
\end{thm}

\begin{proof}
The main theorem of \cite{Seress} states that $b(G) \leqs 4$ for all primitive solvable groups (note that every such group is affine). Similarly, if $|G_0|$ is indivisible by $p$ (the coprime affine case) then $b(G) \leqs 95$ by the main result in \cite{GM}. (In fact, a very recent theorem of Halasi and Podoski \cite{HP} shows that $b(G) \leqs 3$ in this case.) Finally, if $G_0$ is a primitive subgroup of ${\rm GL}(V)$ (that is, $G_0$ does not preserve any nontrivial direct sum decomposition of $V$) then a strong form of Pyber's conjecture (with an explicit constant) is proved in \cite{LS}.
\end{proof}

\subsection{$m$-partitions vs $2$-partitions} \label{ss:part}

Let $G$ be a permutation group on a finite set $\Omega$ and let $(X_0,\ldots,X_{m-1})$ be an $m$-part partition (or \emph{$m$-partition}) of $\Omega$. The stabilizer in $G$ of this partition is defined to be the
intersection of the setwise stabilizers of the $X_i$. Note that the intersection of the setwise stabilizers of the first $m-1$ parts already stabilizes $X_{m-1}$. In particular, the
stabilizer of a $2$-partition is simply the setwise stabilizer of the
first part. 

\begin{prop}\label{m vs 2}
Let $G$ be a permutation group on a finite set $\Omega$ and let $(X_0,\ldots,X_{m-1})$ be an $m$-partition of $\Omega$. Then
there exist $\lceil \log m \rceil$ $2$-partitions of $\Omega$ such that the
stabilizer of $(X_0,\ldots,X_{m-1})$ in $G$ is the intersection of the
stabilizers of these $2$-partitions.
\end{prop}

\begin{proof}
Let $j \in \{ 0,\ldots, m-1 \}$ and let
$$j  = a_0(j)2^0 + a_1(j)2^1 + \cdots + a_d(j)2^d$$
be the binary expansion of $j$, where $d=\lfloor \log m \rfloor$. For each $i \in \{0, \ldots, d\}$ we set
$$\pi_{i,0} = \{ j \in \{ 0,\ldots, m-1 \} \mid a_i(j)=0\},\; 
\pi_{i,1} = \{j \in \{ 0,\ldots, m-1 \} \mid a_i(j)=1\}$$
and
$$\sigma_{i,0} = \bigcup_{j \in \pi_{i,0}}X_j,\;\; \sigma_{i,1} = \bigcup_{j \in \pi_{i,1}}X_j.$$
Let $\sigma_{i}=(\sigma_{i,0}, \sigma_{i,1} )$ and define $\mathcal{P} = \{\sigma_0, \ldots, \sigma_{d}\}$
if $m>2^d$, and $\mathcal{P} = \{\sigma_0, \ldots, \sigma_{d-1}\}$ if $m=2^d$. Note that $|\mathcal{P}| = \lceil \log m \rceil$, and each $\sigma_i \in \mathcal{P}$ is a $2$-partition of $\Omega$.

Clearly, if $g \in G$ stabilizes $(X_0,\ldots,X_{m-1})$ then $g$ stabilizes each of the $2$-partitions in $\mathcal{P}$. For the converse, suppose $g \in G$ stabilizes each partition in $\mathcal{P}$, but for some $j$ there exists a point $\a \in X_j$ such that $\a^g \not\in X_j$, say $\a^g \in X_k$ with $j \neq k$. Then $a_i(j) \neq a_i(k)$ for some $i$, say $a_{i}(j)=0$ and $a_i(k)=1$. But this implies that $\a \in \s_{i,0}$ and $\a^g \in \s_{i,1}$, which is a contradiction since $g$ stabilizes $\s_i$. 
\end{proof}

\subsection{Distinguishing number}\label{ss:dn}

Let $G$ be a permutation group on a finite set $\Omega$. The \emph{distinguishing number} of $G$, denoted by $D(G)$, is defined to be the smallest number of parts in a partition of $\Omega$ with the property that only the identity fixes every part. 
Such a partition is called a \emph{distinguishing partition}. 
Note that $D(G)=1$ if and only if $G$ is trivial, and $D(G)=2$ if and only if $G$ has a regular orbit on the power set of $\Omega$. We also note that $D(S_n)=n$ and $D(A_n)=n-1$, with respect to the natural actions of degree $n$.

\begin{thm}\label{t:dolfi}
Let $G$ be a primitive permutation group of degree $n$, and assume $G \neq A_n,S_n$.
Then $D(G) \leqs 4$.
\end{thm}

\begin{proof}
By a theorem of Seress \cite{Akos1}, there are exactly $43$ possibilities for $G$ with $D(G)>2$. In a later paper, Dolfi \cite{Dolfi} showed that $D(G) \leqs 4$ in each of the exceptional cases. 
\end{proof}

\begin{cor}\label{c:dolfi}
Let $G$ be a primitive permutation group of degree $n$, and assume $G
\neq A_n,S_n$ if $n \geqs 7$. Then there exist three $2$-partitions of
$[n]$ such that the intersection of the stabilizers of these
partitions is trivial.
\end{cor}

\begin{proof}
By Theorem~\ref{t:dolfi}, $D(G) \leqs 6$. Hence, by Proposition~\ref{m
  vs 2}, there exist $3=\lceil \log 6 \rceil$ 2-partitions with the
required property.
\end{proof}

\section{Product-type groups}\label{s:pa}

In view of Theorems \ref{t:as} and \ref{t:d}, in order to prove Theorem \ref{tt:main} we may assume that $G$ is either a product-type group (as in Case III(b) in Table \ref{t:onan}), or a twisted wreath product (Case III(c) in Table \ref{t:onan}). As we will see in Section \ref{s:tw}, the result for twisted wreath products is an easy corollary of the corresponding result for product-type groups, which is the case we focus on in this section.

To get started, let us recall the general set-up for product-type groups. Referring to Table \ref{t:onan}, let $H \leqs {\rm Sym}(\Gamma)$ be a primitive group of type II (almost simple) or III(a) (diagonal). Let $k \geqs 2$ be an integer and consider the wreath product $W = H \wr S_{k}$. This group has a natural product action on the Cartesian product $\Omega=\Gamma^k$, given by
\begin{equation}\label{e:prodd}
(\gamma_1, \ldots, \gamma_{k})^{(h_1, \ldots, h_{k})p^{-1}} = (\gamma_{1^{p}}^{h_{1^{p}}}, \ldots, \gamma_{k^{p}}^{h_{k^{p}}}).
\end{equation}
Let $T={\rm Soc}(H)$ and $B = {\rm Soc}(W)$, so $B=T^k$.
Following \cite{LPS}, a subgroup $G \leqs W$ acting on $\Omega$ is a primitive product-type group if
\begin{itemize}\addtolength{\itemsep}{0.5\baselineskip}
\item[(i)] $B \leqs G$; and
\item[(ii)] $G$ induces a transitive group $P \leqs S_k$ on the $k$ factors of $T^k$.
\end{itemize}
In particular, note that 
$${\rm Soc}(G) = T^k \leqs G \leqs H \wr P.$$

Our proof of Pyber's conjecture for product-type groups is based on the following theorem (recall that $[k]=\{1, \ldots, k \}$). 

\begin{thm}\label{t:main}
Let $P \leqs S_k$ be a transitive permutation group. Then there exist
$$O\left(1+\frac{\log |P|}{k}\right)$$
$2$-partitions of $[k]$ with the property that the intersection of the stabilizers of these partitions in $P$ is trivial.
\end{thm}

Note that this result is essentially best possible. For example, if $P=S_k$ then it is easy to see that at least $\lceil \log k\rceil = O((\log |P|)/k)$ $2$-partitions are required.

The proof of Theorem \ref{t:main} is given in Sections \ref{ss:prim} and \ref{ss:imprim}, where we deal separately with the primitive and imprimitive cases.  In Section \ref{ss:pybpt}, we will use Theorem \ref{t:main} to establish Pyber's conjecture for all primitive product-type groups.

\subsection{The primitive case}\label{ss:prim}

\begin{prop}\label{p:prim}
The conclusion to Theorem \ref{t:main} holds if $P$ is primitive.
\end{prop}

\begin{proof}
If $k \leqs 6$, or if $k \geqs 7$ and $P \neq A_k,S_k$, then Corollary \ref{c:dolfi} states
three $2$-partitions of $[k]$ are enough. 

Now assume $k \geqs 7$ and $P=A_k$ or $S_k$. By applying
Proposition \ref{m vs 2} for the partition of $[k]$ into $1$-element
sets, we deduce that there exist $\lceil \log k \rceil$ 2-partitions
of $[k]$ such that the intersection of the stabilizers of these
partitions is trivial. 
Moreover, since $P=A_k$ or $S_k$ we have 
$$\frac{\log |P|}{k} \geqs c \log k$$
for some absolute constant $c$, so 
$$ \lceil \log k \rceil  = O\left(1+\frac{\log |P|}{k}\right)$$
as required.
\end{proof}

\subsection{The imprimitive case}\label{ss:imprim}

To complete the proof of Theorem \ref{t:main}, we may assume $P \leqs S_k$ is transitive but imprimitive. This situation is rather more difficult than the primitive case handled in the previous section. We begin by introducing some new notation and terminology that we will use throughout this section.

\vs

\noindent \textbf{Structure trees.} Set $\Delta=[k]$. Following the proof of \cite[Theorem 4.1]{GSS} we fix a \emph{structure tree} 
encoding the action of $P$ on $\Delta$. This is a rooted tree $T$ with levels $T_0, T_1, \ldots, T_s$, where the root is $T_0 = \{\Delta\}$ and the leaves are the points in $\Delta$ (that is, $T_s = \Delta$). We set $|T_i|=a_i$, so $a_0=1$ and $a_s=k$. The action of $P$ on $\Delta$ can be extended naturally to the structure tree. We require that the vertices on a fixed level of $T$ define a partition of $\Delta$ into a block system, so $P$ acts transitively on each level of $T$. We also require that if $x \in T_i$ is a non-leaf vertex with children $\Delta(x) \subseteq T_{i+1}$ then $\Delta(x)$ is a partition of $x$, and $P(x):=P_x^{\Delta(x)}$ is a primitive group (that is, the setwise stabilizer of $x$ in $P$, denoted by $P_x$, acts primitively on $\Delta(x)$). Note that if $x,y \in T_i$ then $|\Delta(x)| = |\Delta(y)|$ and the induced groups $P(x)$ and $P(y)$ are permutation isomorphic. For $x \in T_i$ (with $i<s$) we set $|\Delta(x)|=m_{i+1}$, so $a_{i} = \prod_{j=1}^{i}m_j$ for all $i \in \{0,1, \ldots, s\}$.

\vs

\noindent \textbf{Large levels.}  Consider a level $T_i$ in $T$ with $i>0$. Now $P$ induces a transitive permutation group 
$$P^{T_i} \leqs S_{m_i} \wr S_{a_{i-1}} \leqs S_{a_i}$$ 
on the $a_i = m_ia_{i-1}$ vertices at level $T_i$. We will say that $T_i$ is a \emph{large} level of $T$ if $m_i \geqs 7$ and $P(x)=A_{m_i}$ or $S_{m_i}$ for some (hence all) $x \in T_{i-1}$. 

Assume $T_i$ is large, and set 
$$B_i = {\rm Soc}(S_{m_i} \wr S_{a_{i-1}}) = C_1 \times C_2 \times \cdots \times C_{a_{i-1}} = (A_{m_i})^{a_{i-1}}.$$ 
Since $A_{m_i} \leqs P(x)$ for all $x \in T_{i-1}$, it follows that $P^{T_{i}} \cap B_i$ is a subdirect product of $B_i$, so $P^{T_{i}} \cap B_i = \prod_{j}D_j$ is a direct product and each $D_j \cong A_{m_i}$ is a diagonal subgroup of a subproduct $\prod_{\ell \in I_{j}}C_{\ell}$, where the subsets $I_{j}$ form a partition of $[a_{i-1}]$ (see \cite[p.328, Lemma]{Scott}). Since $P$ acts transitively on the $a_{i-1}$ vertices in level $T_{i-1}$, it follows that there exists a divisor $t_i$ of $a_{i-1}$ such that $|I_{j}|=t_{i}$ for all $j$. We call $t_i$  the \emph{linking factor} of the level $T_i$ (note that $P^{T_{i}}$ contains the full direct product $B_i$ if and only if $t_i=1$). Therefore, 
$$P^{T_{i}} \cap B_i = \prod_{j=1}^{a_{i-1}/t_{i}}D_j$$
and
$$D_j  = \{(z,z^{\a_{j,1}}, \ldots, z^{\a_{j,t_i-1}}) \mid z \in A_{m_i}\} \cong A_{m_i}$$ 
is a diagonal subgroup of the direct product $\prod_{\ell \in I_{j}}C_{\ell} = (A_{m_i})^{t_i}$, with $\a_{j,\ell} \in {\rm Aut}(A_{m_i})$, $1 \leqs \ell \leqs t_{i}-1$. Note that ${\rm Aut}(A_{m_i}) = S_{m_i}$ since $m_i \geqs 7$. The large levels of $T$ will require special attention in the proof of Theorem \ref{t:main}. 

\vs

We start by considering two special cases, which provide the basis for the general argument given in Proposition \ref{p:3}. For the remainder of this section we will freely adopt the notation and terminology introduced above. 

\begin{prop}\label{p:1}
The conclusion to Theorem \ref{t:main} holds if $T$ has no large levels. In fact, in this situation at most six partitions suffice.
\end{prop}

\begin{proof}
Let $x \in T$ be a non-leaf vertex and recall that $P(x)$ acts primitively on $\Delta(x)$. By Corollary \ref{c:dolfi}, there exist three $2$-partitions 
$$\Delta(x) = \Delta_j(x) \cup (\Delta(x) \setminus \Delta_j(x)), \; 1 \leqs j \leqs 3$$
with the property that the intersection of the stabilizers of these
partitions in $P(x)$ is trivial. We now proceed as in the proof of
\cite[Theorem 4.1]{GSS} and \cite[Theorem 1.2]{Seress}.

We inductively define three $3$-colorings of the vertices of $T$, denoted $F_j:T \to \mathbb{F}_{3}$, $1 \leqs j \leqs 3$. First, color the root vertex $0$. Suppose that the $F_j$-coloring of the levels $T_0, \ldots, T_i$ has already been defined and let $x \in T_i$. Then for $y \in \Delta(x)$ we define 
\begin{equation}\label{e:fj}
F_j(y) = \left\{\begin{array}{ll}
F_j(x) & \mbox{if $y \in \Delta_j(x)$} \\
F_j(x)+1 & \mbox{if $y \in \Delta(x) \setminus \Delta_j(x)$}. 
\end{array}\right.
\end{equation}

Suppose $p \in P$ fixes all three colorings of $T$. By induction on $i=0,1,\ldots, s$, a color-preserving permutation of $T_0 \cup T_1 \cup \cdots \cup T_i$ must fix $T_0 \cup T_1 \cup \cdots \cup T_i$ pointwise. In particular, $p$ fixes $T_s=\Delta$ pointwise, so $p=1$. 

Next observe that each $F_j$-coloring of the entire tree $T$ can be reconstructed from the corresponding coloring of the leaves. Moreover, if $p \in P$ fixes the $F_j$-coloring of the leaves, then by induction on $i=s,s-1, \ldots, 0$, we see that $p$ must also fix the $F_j$-coloring of $T_i \cup T_{i+1} \cup \cdots \cup T_{s}$. Therefore, if $p$ fixes all three colorings of the leaves then $p$ fixes all three colorings of $T$, so $p=1$ by the previous argument. 

Now each $3$-coloring of the leaves corresponds to a $3$-partition of $\Delta$, say $(X_1,X_2,X_3)$. By Proposition \ref{m vs
  2}, there exist $2=\lceil \log 3 \rceil$ $2$-partitions of $\Delta$ such that
the intersection of their stabilizers also stabilizes $(X_1,X_2,X_3)$.
Therefore, we can define six $2$-partitions of $\Delta$ so that if $p$ stabilizes each of these partitions then $p$ must preserve all three $F_j$-colorings of the leaves of $T$, and by the above remarks this implies that $p=1$.
\end{proof}

\begin{prop}\label{p:2}
The conclusion to Theorem \ref{t:main} holds if $T$ has a unique large level.
\end{prop}

\begin{proof}
Let $T_{\ell}$ be the unique large level of $T$, where $1 \leqs \ell \leqs s$. Let $t$ be the corresponding linking factor. As in the proof of the previous proposition, we start by defining three \emph{global} colorings of $T$, denoted $F_j:T \to \mathbb{F}_{3}$, $1 \leqs j \leqs 3$. These colorings are defined inductively as in \eqref{e:fj} (starting at the root, setting $F_j(T_0)=0$ for all $j$), except that if $x \in T_{\ell-1}$ then we set $\Delta_j(x) = \Delta(x)$ for all $j$ (that is, all the children of $x$ inherit the coloring of $x$).

Suppose $p \in P$ fixes all three colorings of $T$. Then by induction on $i=0,1, \ldots, \ell-1$, $p$ must fix $T_0 \cup T_1 \cup \cdots \cup T_{\ell-1}$ pointwise. We also note that each $F_j$-coloring of the entire structure tree $T$ can be reconstructed from the corresponding coloring of the leaves. Moreover, if $p \in P$ fixes all three $F_j$-colorings of the leaves then $p$ fixes all three global colorings of $T$, which implies that $p$ fixes $T_0 \cup T_1 \cup \cdots \cup T_{\ell-1}$ pointwise. Therefore, as in the previous proposition, we can define six $2$-partitions of $\Delta$ with the property that if $p \in P$ stabilizes all six partitions then $p$ fixes $T_0 \cup T_1 \cup \cdots \cup T_{\ell-1}$ pointwise. In addition, by considering the $F_j$-colorings of $T_{\ell+1} \cup \cdots \cup T_{s}$, we note that if $p$ also fixes the large level $T_{\ell}$ pointwise 
then $p$ fixes every vertex in the entire tree, so $p=1$.

Consider the large level $T_{\ell}$. To simplify notation, set $a=a_{\ell-1}$ and $m=m_{\ell}$, so $|T_{\ell-1}|=a$, $|T_{\ell}|=ma$ and $m \geqs 7$. We will describe a bijection between the set of colorings of the vertices in $T_{\ell}$ with $2^{k/am}$ colors, and the set of $2$-colorings of $\Delta$. 

Let $z \in T_{s-1}$. 
The subset $\Delta(z) \subseteq \Delta$ has $m_{s}$ elements, so it has $2^{m_{s}}$ subsets. Each choice of subset $J \subseteq \Delta(z)$ yields a $2$-coloring of $\Delta(z)$: use $0$ to color the points in $J$, and $1$ for the points in $\Delta(z) \setminus J$. In this way, by choosing such subsets for all $\Delta(z)$ with $z \in T_{s-1}$ we obtain a $2$-coloring of $\Delta$. Conversely, if $T_{s-1}=\{z_1, \ldots, z_{a_{s-1}}\}$ then any $2$-coloring of $\Delta$ corresponds to a collection of subsets $\{J_1, \ldots, J_{a_{s-1}}\}$ with $J_i \subseteq \Delta(z_i)$.  We can now color each vertex of $T$ at level $T_{s-1}$ with one of $2^{m_{s}}$ colors; the color of $z \in T_{s-1}$ being uniquely determined by the given $2$-coloring of $\Delta(z)$. 

Next consider an element $y \in T_{s-2}$. Here $|\Delta(y)|=m_{s-1}$, and each child of $y$ can take one of $2^{m_{s}}$ colors. Therefore, there are $(2^{m_{s}})^{m_{s-1}} = 2^{m_{s}m_{s-1}}$ possible colorings of $\Delta(y)$, so we can color the vertex $y$ at level $T_{s-2}$ with one of $2^{m_{s}m_{s-1}}$ colors according to the coloring of $\Delta(y)$. Continuing in this way, working up through the levels in $T$, we see that we can color the $ma$ vertices at level $T_{\ell}$ with $2^{\prod_{i>\ell}m_{i}} = 2^{k/am}$ colors (note that $k=am$ if $\ell=s$). By construction, the coloring of $T_{\ell} \cup \cdots \cup T_{s}$ is uniquely determined by the $2$-coloring of the leaves. Conversely, note that any coloring of $T_{\ell}$ with $2^{k/am}$ colors yields a unique coloring of each lower level $T_{i}$ ($i>\ell$) with $2^{\prod_{j>i}m_j}$ colors. In particular, any such coloring of $T_{\ell}$ induces a unique $2$-coloring 
of $\Delta$.

Set $\chi = 2^{k/am}$. We need to determine how many colorings of
$T_{\ell}$ are needed (using $\chi$ colors) so that if $x \in
T_{\ell-1}$ and $p \in P_x$ (the setwise stabilizer) fixes each
coloring of $T_{\ell}$ then $p$ acts trivially on $\Delta(x)$. To do
this we mimic the proof of Proposition \ref{m vs 2}, working with 
base-$\chi$ expansions, rather than base-$2$. Let $j \in \{ 0,\ldots, m-1 \}$ and let
\begin{equation}\label{e:aij}
j  = a_0(j)\chi^0 + a_1(j)\chi^1 + \cdots + a_d(j)\chi^d
\end{equation}
be the base-$\chi$ expansion of $j$, where $d=\lfloor (am/k)\log m
\rfloor$. For each $i \in \{0, \ldots, d\}$ and each $f \in \{ 0,
\ldots, \chi-1 \}$ we set
$$\pi_{i,f} = \{ j \in \{ 0,\ldots, m-1 \} \mid a_i(j)=f\}$$
and 
\begin{equation}\label{e:pii}
\pi_i = (\pi_{i,0}, \pi_{i,1}, \ldots, \pi_{i,\chi-1}),
\end{equation}
so $\{\pi_0, \ldots, \pi_{d}\}$ is a collection of $\chi$-partitions
of $\{ 0,\ldots, m-1 \}$ if $m>\chi^d$, and if $m=\chi^d$ then $\{\pi_0,
\ldots, \pi_{d-1}\}$ is such a collection. To simplify the notation in the
rest of the proof, we define $\pi_d:=\pi_{d-1}$ in the latter case, so that in both cases we can refer to $d+1$ $\chi$-partitions $\{\pi_0, \ldots, \pi_{d}\}$. Note that if $\rho$ is a permutation of $\{0,\ldots, m-1\}$ that stabilizes each of these partitions then $a_i(j)=a_i(j^{\rho})$ for all $i,j$ (see \eqref{e:aij}), so  $\rho=1$.

Recall that $T_{\ell}$ is large, with linking factor $t$ dividing $a$. Set 
$$B_{\ell} = {\rm Soc}(S_{m} \wr S_a) = C_1 \times C_2 \times \cdots \times C_{a} = (A_m)^a$$
and recall that
\begin{equation}\label{e:dec}
P^{T_{\ell}} \cap B_{\ell} = D_1 \times D_2 \times \cdots \times D_{a/t}
\end{equation}
is a subdirect product of $B_{\ell}$, where each
\begin{equation}\label{e:dj}
D_i  = \{(z,z^{\a_{i,1}}, \ldots, z^{\a_{i,t-1}}) \mid z \in A_{m}\} \cong A_{m}
\end{equation}
is a diagonal subgroup of a direct product $\prod_{j \in I_{i}}C_j = (A_{m})^{t}$ with $\a_{i,j} \in {\rm Aut}(A_m) = S_m$ (since $m \geqs 7$). Here the $I_i$ form a partition of $[a]$, and to simplify notation we will assume that $I_{i} = \{(i-1)t+1, \ldots, (i-1)t+t\}$ for all $1 \leqs i \leqs a/t$.  

Write $T_{\ell-1} = \{y_0, \ldots, y_{a-1}\}$ and consider the partition 
\begin{equation}\label{e:pp}
T_{\ell} = \Delta(y_0) \cup \cdots \cup \Delta(y_{a-1})
\end{equation}
of $T_{\ell}$ into $a$ children sets of size $m$. 
Write $d+1 = bt+r$ with $0 \leqs r<t$ and set $e = \lceil (d+1)/t \rceil -1$.
We will define a collection $\{\s_0, \ldots, \s_e\}$ of $\chi$-colorings of
$T_{\ell}$ in terms of the partition of $T_{\ell}$ in \eqref{e:pp},
and the structure of $P^{T_{\ell}} \cap B_{\ell}$ presented in
\eqref{e:dec} and \eqref{e:dj}. 

For $j \in \{0, \ldots, e-1\}$ we define a $\chi$-coloring $\s_j$ of $T_{\ell}$ as follows. Fix $i \in \{0,\ldots, a-1\}$ and write $i = qt+v$ with $0 \leqs v<t$. We may arbitrarily identify $\Delta(y_{qt})$ with the set of integers $\{0, \ldots, m-1\}$, and so the $\chi$-partition $\pi_{jt+v}$ defined above (see \eqref{e:pii}) corresponds to a $\chi$-coloring of $\Delta(y_{qt})$. In \eqref{e:dj}, $\a_{q+1,v} \in {\rm Aut}(A_m) = S_m$ defines a bijection between $\Delta(y_{qt})$ and $\Delta(y_i)$ (we set $\a_{k,0}=1$ for all $1 \leqs k \leqs a/t$), so we can view $(\pi_{jt+v})^{\a_{q+1,v}}$ as a $\chi$-partition, and thus a $\chi$-coloring, of $\Delta(y_i)$. This defines a $\chi$-coloring of $\Delta(y_i)$, and we repeat the process for all $i \in \{0, \ldots, a-1\}$ to obtain a $\chi$-coloring of $T_{\ell}$, denoted by $\s_j$.

For example, the colorings $\s_0$ and $\s_1$ are defined as follows:

\vspace{-3mm}

$$\begin{array}{c|clll|clll|l}
& \Delta(y_0) & \Delta(y_1) & \ldots & \Delta(y_{t-1}) & \Delta(y_{t}) & \Delta(y_{t+1}) & \ldots & \Delta(y_{2t-1}) & \ldots \\ \hline
\s_0 & \pi_0 & (\pi_1)^{\a_{1,1}} & \ldots & (\pi_{t-1})^{\a_{1,t-1}} & \pi_0 & (\pi_1)^{\a_{2,1}} & \ldots & (\pi_{t-1})^{\a_{2,t-1}} & \ldots \\
\s_1 & \pi_t & (\pi_{t+1})^{\a_{1,1}} & \ldots & (\pi_{2t-1})^{\a_{1,t-1}} & \pi_t & (\pi_{t+1})^{\a_{2,1}} & \ldots & (\pi_{2t-1})^{\a_{2,t-1}} & \ldots 
\end{array}$$

If $r=0$ then $\s_e$ is defined by coloring $\Delta(y_i)$ with $(\pi_{et+v})^{\a_{q+1,v}}$ as before. However, if $r>0$ then we define $\s_e$ by 
coloring $\Delta(y_i)$ with $(\pi_{et+v})^{\a_{q+1,v}}$ if $v<r$, and with $(\pi_{et+r-1})^{\a_{q+1,v}}$ if $v \geqs r$ (note that $d= et+r-1$ if $r>0$).

\vs

\noindent \textbf{Claim 1.} \emph{Suppose $x \in T_{\ell-1}$ and $p \in P_x$ fixes each of the colorings $\s_0, \ldots, \s_e$ of $T_{\ell}$. Then $p$ fixes every vertex in $\Delta(x)$.}

\begin{proof}[Proof of claim]
To see this, first assume $x=y_0$ and let $i \in [t-1]$. Given the structure of $P^{T_{\ell}}$ in 
\eqref{e:dec}, it follows that $p$ also fixes $y_i$ setwise. More precisely, if $z \in {\rm Sym}(\Delta(y_0))$ denotes the action of $p$ on $\Delta(y_0)$ then the description of the diagonal subgroup $D_1$ in \eqref{e:dj} indicates that the action of $p$ on $\Delta(y_i)$ is given by $z^{\a_{1,i}}$.
We need to show that $z$ is trivial.

By definition of $\s_0$, we immediately deduce that $z$ fixes the $\chi$-coloring $\pi_0$ of $\Delta(y_0)$. Moreover, we see that $z^{\a_{1,i}} \in {\rm Sym}(\Delta(y_i))$ fixes the $\chi$-coloring $(\pi_i)^{\a_{1,i}}$ of $\Delta(y_i)$. In other words, if $p$ fixes the coloring $\s_0$ of $T_{\ell}$ then $z$ fixes the $\chi$-colorings $\pi_0, \pi_{1}, \ldots, \pi_{t-1}$ of $\Delta(y_0)$. 

Similarly, by considering $\s_1$, we see that $z$ also
fixes the $\chi$-colorings $\pi_t, \pi_{t+1}, \ldots, \pi_{2t-1}$ of $\Delta(y_0)$. Continuing in this way, we deduce that $z$ must fix all of the colorings $\{\pi_0,
\ldots, \pi_d\}$ of $\Delta(y_0)$, whence $a_i(j)=a_i(j^z)$ for all $i$
and all $j \in \Delta(y_0) $ (see \eqref{e:aij}) and thus $z$ is trivial, as required. 

The same argument applies if $x=y_{qt}$ for any
$q \in \{ 1,\ldots, a/t-1\}$. Finally, if $x=y_{qt+v}$ with $v>0$ then the claim follows from the fact that the actions of $p$ on $\Delta(y_{qt})$ and
$\Delta(y_{qt+v})$ are permutation isomorphic.
\end{proof}

As previously noted, every $\chi$-coloring of $T_{\ell}$ corresponds to a unique $2$-coloring of $\Delta$, so from the $\chi$-colorings $\s_0, \ldots, \s_e$ we obtain $e+1$ $2$-colorings of $\Delta$. Also, recall that we have defined an additional six $2$-colorings of $\Delta$, which arise from the three global $F_j$-colorings of $T$ defined at the beginning of the proof.

\vs

\noindent \textbf{Claim 2.} \emph{Suppose $p \in P$ fixes all $e+7$ $2$-colorings of $\Delta$. Then $p=1$.}

\begin{proof}[Proof of claim]
In view of our earlier remarks, in order to justify this claim it
suffices to show that if $p$ fixes each of the original six
$2$-colorings of $\Delta$, and $p$ also fixes the $2$-coloring $\tau$ of
$\Delta$ corresponding to any one of the above $\chi$-colorings $\s \in \{\s_0, \ldots, \s_e\}$ at level $T_{\ell}$, then $p$ must fix the coloring $\s$ of level $T_{\ell}$. (Recall that we have already noted that such an element $p$ fixes $T_0 \cup \cdots \cup T_{\ell-1}$ pointwise.) This is clear if $\ell=s$, so let us assume $\ell<s$. We proceed by induction on $i=s,s-1, \ldots, \ell$. 

For $i \geqs \ell$, let $\s(i)$ denote the unique coloring of $T_{i}$ induced
from the $\chi$-coloring $\s$ of level $T_{\ell}$. Suppose $p \in P$
fixes the original six $2$-colorings of $\Delta$, and also the
$2$-coloring $\tau = \s(s)$ of $\Delta$. We claim that $p$ fixes the 
$2^{m_s}$-coloring $\s(s-1)$ of $T_{s-1}$. Seeking a contradiction,
suppose $p$ does not fix this coloring, say $u,v \in T_{s-1}$ have
different colors, and $p$ maps $u$ to $v$. For this to happen, there
must be an element $q \in P(u)$ (the primitive group on $\Delta(u)$
induced by $P_u$) that sends the subset of $\Delta(u)$ corresponding
to the color of $u$ to the subset corresponding to the color of
$v$. Moreover, $q$ must fix the original three $F_j$-colorings of
$\Delta(u)$ (since $p$ fixes the $F_j$-colorings of $\Delta$). But by construction, the identity is the only element in $P(u)$ that fixes the three $F_j$-colorings of $\Delta(u)$. This is a contradiction, hence $p$ fixes the coloring $\s(s-1)$ of $T_{s-1}$. In the same way, we see that $p$ fixes the $2^{m_sm_{s-1}}$-coloring $\s(s-2)$ of $T_{s-2}$, and so on. In particular, working our way up the levels of $T$, we deduce that $p$ fixes the $\chi$-coloring $\s(\ell)=\s$ of $T_{\ell}$. The claim follows.
\end{proof}

In view of Claim 2, we have now found $e+7$ $2$-partitions of $\Delta$ with the desired trivial intersection property. Now
$$e+7 = \left\lceil \frac{d+1}{t} \right\rceil+6 \leqs \frac{am\log m}{kt}+8$$
and by considering the permutation group $P^{T_{\ell}}$ induced by $P$ on $T_{\ell}$ we deduce that
$$|P| \geqs |A_m|^{\frac{a}{t}} \geqs m^{\frac{cam}{t}}$$
for some absolute constant $c$. Therefore, 
$$e+7 \leqs c_1+c_2\frac{\log |P|}{k}$$
for some absolute constants $c_1$ and $c_2$. This completes the proof of Proposition \ref{p:2}.
\end{proof}

\begin{prop}\label{p:3}
The conclusion to Theorem \ref{t:main} holds if $P$ is imprimitive.
\end{prop}

\begin{proof}
We may assume that $T$ has at least two large levels. Set $\mathcal{L} = \{i \mid \mbox{$T_i$ is large}\} = \{\ell_1, \ldots, \ell_n\}$, where $\ell_j<\ell_{j+1}$ for all $j$. Let $t_{i}$ be the linking factor of $T_{\ell_i}$. 

As before, we start by defining three global colorings of $T$, denoted $F_j:T \to \mathbb{F}_{3}$, $1 \leqs j \leqs 3$. These are  defined inductively as in \eqref{e:fj} (starting with $F_j(T_0)=0$), except that if $x \in T_{\ell_i-1}$ (for any $1 \leqs i \leqs n$) then we set $\Delta_j(x) = \Delta(x)$ for all $j$. In the usual way, the three $F_j$-colorings of $T$ yield six $2$-colorings of $\Delta$. Moreover, if $p \in P$ fixes all six $2$-colorings of $\Delta$ then $p$ fixes $T_0 \cup T_1 \cup \cdots \cup T_{\ell_1-1}$ pointwise.

For each large level $T_{\ell_i}$ in $T$ we define a collection of $e_{i}$ colorings of $T_{\ell_i}$ in $\chi_i=2^{\prod_{j>\ell_i}m_j}$ colors. These colorings are defined exactly as in the proof of the previous proposition, so that 
$$e_{i} = \left\lceil \frac{d_{i}+1}{t_i}\right\rceil \mbox{ and } d_{i} = \left\lfloor \left(\prod_{j>\ell_i}m_j^{-1}\right) \log m_{\ell_i} \right\rfloor$$
and each of these $\chi_i$-colorings yields a specific $2$-coloring of $\Delta$. 
In particular, if $x \in T_{\ell_i-1}$ and $p \in P_x$ fixes all $e_{i}$ colorings of $T_{\ell_i}$ then $p$ acts trivially on $\Delta(x)$ (see Claim 1 in the proof of Proposition \ref{p:2}). In this way, we end up with a collection of
$$\a=6+\sum_{i=1}^{n}e_{i}$$
$2$-colorings (or $2$-partitions) of $\Delta$. 

To complete the proof of the proposition, it is sufficient to show that
\begin{itemize}\addtolength{\itemsep}{0.5\baselineskip}
 \item[(i)] $\a \leqs c_1+c_2\frac{\log |P|}{k}$ for some absolute constants $c_1,c_2$; and
\item[(ii)] if $p \in P$ fixes all $\a$ $2$-colorings of $\Delta$ then $p=1$. 
\end{itemize}

First consider (i). Just by considering the large levels in $T$ we calculate that 
$$|P| \geqs \prod_{i=1}^n |A_{m_{\ell_i}}|^{t_i^{-1}a_{\ell_i-1}} = \prod_{i=1}^n |A_{m_{\ell_i}}|^{t_i^{-1}\prod_{j<\ell_i}m_j}$$
and thus
$$\log |P| \geqs c \sum_{i=1}^n t_i^{-1}\left(\prod_{j=1}^{\ell_i}m_j\right)\log m_{\ell_i}$$
for some absolute constant $c$. In particular, since $k=\prod_{i=1}^{s}m_i$ it follows that 
$$\frac{\log |P|}{k} \geqs c \sum_{i=1}^{n} t_i^{-1}\left(\prod_{j>\ell_i}m_j^{-1}\right)\log m_{\ell_i}.$$
This establishes (i). 

Finally, let us turn to (ii). Suppose $p \in P$ fixes all $\a$ $2$-colorings of $\Delta$. We need to show that $p=1$, and to do this we will use induction.
First consider the 
$e_{n}$ $2$-colorings of $\Delta$ corresponding to the $\chi_n$-colorings of the lowest large level $T_{\ell_n}$. Since $p$ fixes the original six $2$-colorings of $\Delta$, we deduce that $p$ fixes all the $e_{n}$ colorings of level $T_{\ell_n}$ (we repeat the argument given in the proof of Claim 2 in the proof of  Proposition \ref{p:2}). Therefore, if $x \in T_{\ell_n-1}$ then each $p \in P_x$ fixes every vertex in $\Delta(x)$. 

Next consider the $e_{n-1}$ $2$-colorings of $\Delta$ induced from the $\chi_{n-1}$-colorings of the vertices at level $T_{\ell_{n-1}}$. Since we are assuming that $p$ fixes the $6+e_{n}$ $2$-colorings of $\Delta$ that arise from the $F_j$-colorings of $T$ and the $\chi_n$-colorings of $T_{\ell_n}$, we deduce that $p$ fixes each of the $e_{n-1}$ colorings of $T_{\ell_{n-1}}$. In particular, if $x \in T_{\ell_{n-1}-1}$ then each $p \in P_x$ fixes every vertex in $\Delta(x)$. Continuing in this way, we deduce that if $p \in P_x$ for any $x \in T_{\ell_i-1}$ then $p$ fixes every vertex in $\Delta(x)$. Of course, the same property holds if $x \in T_{i-1}$ and $T_i$ is a non-large level (that is, each $p \in P_x$ acts trivially on $\Delta(x)$) because $p$ fixes the original three $F_j$-colorings of $T_{i}$. 

We have already noted that $p$ fixes $T_0 \cup \cdots \cup T_{\ell_1-1}$ pointwise. Working inductively, this time starting at the root of $T$ and working down level-by-level, we deduce that $p$ fixes all vertices in $T_{\ell_1}$. Therefore, by the argument above, $p$ fixes all vertices at level $T_{\ell_1+1}$, and so on. Continuing in this way, we deduce that $p$ fixes all the vertices in $T_s$, that is, $p$ fixes $\Delta$ pointwise, so $p=1$ as required. This justifies (ii), and the proof of the proposition is complete.
\end{proof}

This completes the proof of Theorem \ref{t:main}.

\subsection{Pyber's conjecture}\label{ss:pybpt}

In this section we will apply Theorem \ref{t:main} to establish Pyber's conjecture for product-type groups. 

\begin{thm}\label{t:pybppa}
Pyber's conjecture holds if $G$ is a product-type group.
\end{thm}

\begin{proof}
First, let us recall the general set-up for product-type groups. We have $\Omega = \Gamma^k$ for some
set $\Gamma$ and integer $k \geqs 2$, and referring to the notation in Table \ref{t:onan}, there exists a primitive group $H \leqs {\rm Sym}(\Gamma)$ of type {\rm II} (almost simple) or type {\rm
  III(a)} (diagonal-type) with socle $T$, and a transitive subgroup $P
\leqs {\rm Sym}(\Delta)$ (with $\Delta=[k]$) induced by the conjugation action of $G$ on the $k$ factors of $T^k$, such that 
$${\rm Soc}(G) = T^k \leqs G \leqs H \wr P.$$
(Here $p \in P$ if and only if $(h_1, \ldots, h_k)p \in G$ for some $h_i \in H$.)
Note that 
\begin{equation}\label{e:ordd}
|G| \geqs |T|^k|P|.
\end{equation}
Write $\Omega = \Gamma_1 \times \cdots \times \Gamma_k$ where $\Gamma_i=\Gamma$ for each $i$. The action of $G$ on $\Omega$ is described in \eqref{e:prodd}.

\begin{lem}\label{l:kov}
Without loss of generality, we may assume that $G$ induces $H$ on each of the $k$ factors $\Gamma_i$ of $\Omega$.
\end{lem}

\begin{proof}
Set 
$G_1 = \{(h_1, \ldots, h_k)p \in G \mid 1^{p}=1\}$
and let $R \leqs H \leqs {\rm Sym}(\Gamma_1)$ be the permutation group induced by $G_1$ on $\Gamma_1$. By a result of Kov\'{a}cs \cite[(2.2)]{Kovacs},
we may replace $G$ by a conjugate $G^x$ for some $x \in \prod_{i=1}^k{\rm Sym}(\Gamma_i)<{\rm Sym}(\Omega)$ so that $G \leqs R \wr P$ and $G$ induces $R$ on each factor $\Gamma_i$ of $\Omega$. In particular, the induced group $R \leqs {\rm Sym}(\Gamma_i)$ is primitive of type {\rm II} or {\rm III(a)}, according to the type of $H$. Therefore, we may as well assume that $R=H$, and the result follows. 
\end{proof}

By Theorem \ref{t:main}, there exists a collection $\{\pi_1, \ldots, \pi_a\}$ of 
$2$-partitions of $\Delta = [k]$ such that the intersection in $P$ of the stabilizers of these partitions is trivial, and the size of this collection satisfies the bound
\begin{equation}\label{e:aa}
a \leqs c_1 + c_2\frac{\log |P|}{k}
\end{equation} 
for some absolute constants $c_1,c_2$ (which are independent of $P$ and $k$). Set 
\begin{equation}\label{e:rdef}
r = \lfloor \log |\Gamma| \rfloor.
\end{equation}

\begin{lem}\label{l:PP}
Let $a$ and $r$ be the integers in \eqref{e:aa} and \eqref{e:rdef}. Then there exists a collection of points $\{\a_1, \ldots, \a_{\lceil a/r \rceil}\}$ in $\Omega$ with the property that an element $g=(1,\ldots, 1)p \in G$ fixes each $\a_i$ if and only if $p=1$.
\end{lem}

\begin{proof}
Write $a=qr+c$ with $0 \leqs c<r$, and 
partition the set $\{\pi_1, \ldots, \pi_a\}$ into $q$ subsets of size $r$ (plus an additional subset of size $c$ if $c>0$). To simplify the notation, suppose $S=\{\pi_1, \ldots, \pi_r\}$ is one of these subsets. By taking the common refinement of the $2$-partitions $\pi_i$ in $S$ we obtain an $s$-partition $(\s_1, \ldots, \s_s)$ of $[k]$ such that $s \leqs 2^r \leqs |\Gamma|$ and each subset $\s_i$ is contained in one of the two parts of each partition in $S$. Choose distinct elements $\gamma_1, \ldots, \gamma_s \in \Gamma$ (we can do this since $s \leqs |\Gamma|$) and define $\a \in \Omega$ so that all the coordinates in $\a$ corresponding to points in $\s_i$ are equal to $\gamma_i$. Note that if $g=(1,\ldots, 1)p \in G$ fixes $\a$ then $p$ stabilizes all of the $2$-partitions in $S$. 

Continuing in this way, we construct a collection of points 
$\{\a_1, \ldots, \a_{\lceil a/r \rceil}\}$
in $\Omega$ with the property that if $g = (1, \ldots, 1)p \in G$ fixes each $\a_i$ then $p$ fixes each of the $2$-partitions in $\{\pi_1, \ldots, \pi_a\}$. But only the trivial permutation fixes each $\pi_i$, so $p=1$ as required.
\end{proof}

To complete the proof of Theorem \ref{t:pybppa}, we now distinguish two cases, according to whether or not $H$ is almost simple or diagonal. 

\begin{prop}\label{p:pybs}
Pyber's conjecture holds if $H \leqs {\rm Sym}(\Gamma)$ is almost
simple.
\end{prop}

\begin{proof}
Here ${\rm Soc}(H)=T$ is a nonabelian simple group and $|H| \leqs |{\rm Aut}(T)| \leqs |T|^2$, so \eqref{e:ordd} implies that
\begin{equation}\label{e:asbb}
|G| \geqs |T|^{k}|P| \geqs |H|^{\frac{k}{2}}|P|.
\end{equation}
Let $b=b(H)$ be the base size of $H \leqs {\rm Sym}(\Gamma)$. By Theorem \ref{t:as} we have 
\begin{equation}\label{e:bb}
b \leqs c_3\frac{\log |H|}{\log |\Gamma|}
\end{equation}
for some absolute constant $c_3$. Let $\{\gamma_1, \ldots, \gamma_{b}\}$ be a base for $H$ and set $\a_{i}' =(\gamma_i, \ldots, \gamma_i) \in \Omega$ for all $1 \leqs i \leqs b$. Suppose $g = (h_1, \ldots, h_k)p^{-1} \in G$ fixes each $\a_i'$. Since
$$(\gamma_i, \ldots, \gamma_i)^{g} = (\gamma_i^{h_{1^{p}}}, \ldots, \gamma_i^{h_{k^{p}}})$$
it follows that each $h_j$ fixes $\gamma_i$ for all $1 \leqs i \leqs b$, so $h_j=1$ and thus $g=(1,\ldots, 1)p^{-1}$. 

By applying Lemma \ref{l:PP}, we deduce that  
\begin{equation}\label{e:base}
\mathcal{B}=\{\a_1, \ldots, \a_{\lceil a/r \rceil}\} \cup \{\a_1', \ldots, \a_{b}'\}
\end{equation}
is a base for $G$, where $r=\lfloor \log |\Gamma| \rfloor$ (see \eqref{e:rdef}) and $a$ is an integer that satisfies the upper bound in \eqref{e:aa}. In view of \eqref{e:aa} and \eqref{e:bb}, it follows that there exist absolute constants $c_i$ such that
\begin{align*}
b(G) \leqs \lceil a/r \rceil+b & \leqs \left\lceil c_1\frac{1}{\lfloor \log |\Gamma| \rfloor} + c_2\frac{\log |P|}{k\lfloor \log |\Gamma| \rfloor} \right\rceil+c_3\frac{\log |H|}{\log |\Gamma|}\\
& \leqs c_4\frac{\log |P|}{\log |\Omega|} + c_5\frac{\log |H^k|}{\log |\Omega|} \\
& \leqs c_6\frac{\log |G|}{\log |\Omega|}
\end{align*}
as required (the final inequality follows from \eqref{e:asbb}).
\end{proof}

\begin{prop}\label{p:pybd}
Pyber's conjecture holds if $H \leqs {\rm Sym}(\Gamma)$ is diagonal.
\end{prop}

\begin{proof}
Here ${\rm Soc}(H) = T = S^{\ell}$, where $S$ is a nonabelian simple group, $\ell \geqs 2$ and 
$${\rm Soc}(H) = S^{\ell} \leqs H \leqs S^{\ell}.({\rm Out}(S) \times Q)$$
where $Q \leqs S_{\ell}$ is the permutation group induced from the conjugation action of $H$ on the $\ell$ factors of $S^{\ell}$. Since $H$ is primitive, either $Q \leqs S_{\ell}$ is primitive, or $\ell=2$ and $Q=1$.
Note that $|\Gamma| = |S|^{\ell-1}$ and
\begin{equation}\label{e:hordd}
|H| \leqs |S|^{\ell+1}|S_{\ell}|.
\end{equation}
There are two cases to consider.

\vs

\noindent \textbf{Case 1.} \emph{$\ell \leqs 6$ or $A_{\ell} \not\leqs Q$}

\vs

\noindent As in the proof of Proposition \ref{p:pybs}, the set $\mathcal{B}$ in \eqref{e:base} is a base for $G$, where $a$ and $r$ are given in \eqref{e:aa} and \eqref{e:rdef}, respectively, and $b=b(H)$. 
Since we are assuming that either $\ell \leqs 6$ or $A_{\ell} \not\leqs Q$, the main theorem of \cite{Fawcett} (see the proof of Theorem \ref{t:d}) implies that $b \leqs 4$. Therefore, using the bound in \eqref{e:ordd}, we deduce that 
$$b(G) \leqs \lceil a/r \rceil+4  \leqs \left\lceil c_1\frac{1}{\lfloor \log |\Gamma| \rfloor} + c_2\frac{\log |P|}{k\lfloor \log |\Gamma| \rfloor} \right\rceil + 4  \leqs c_3 \frac{\log |G|}{\log |\Omega|}$$
for some absolute constant $c_3$, as required.

\vs

\noindent \textbf{Case 2.} \emph{$\ell \geqs 7$, and $Q=A_{\ell}$ or $S_{\ell}$} 

\vs

\noindent By Lemma \ref{l:kov}, we may assume that $G$ induces $H$ on each of the $k$ factors $\Gamma_i=\Gamma$ in $\Omega=\Gamma^k$. In particular, if we write ${\rm Soc}(G) = T_1 \times \cdots \times T_k$ with $T_i=S^{\ell}$ for each $i$, then $G$ induces $Q$ on the set of $\ell$ simple factors in each factor $T_i$. Let 
\begin{equation}\label{e:X}
X \leqs Q \wr P \leqs S_{k\ell}
\end{equation} 
be the group induced by the conjugation action of $G$ on the $k\ell$ factors of ${\rm Soc}(G)=S^{k\ell}$. 

By the proof of  \cite[Proposition 3.8]{Fawcett}, there exist elements 
$\gamma_1,\gamma_2,\gamma_3 \in \Gamma$ such that the pointwise stabilizer in $H$ of these points is contained in $Q$. (In other words, if $(s_1,\ldots,s_{\ell})q \in H$ fixes each $\gamma_i$ then $s_j=1$ for all $j$.)
Set $\a_i'' = (\gamma_i, \ldots, \gamma_i) \in \Omega$ for $i=1,2,3$, and let $Y$ be the pointwise stabilizer in $G$ of the elements 
\begin{equation}\label{e:aldd}
\{\a_1'',\a_2'',\a_3''\}
\end{equation}
Then we may view $Y$ as a subgroup of $X$, where $X$ is defined in \eqref{e:X}. 

Since $A_{\ell} \leqs Q$ it follows that
$${\rm Soc}(Q \wr P) = C_1 \times C_2 \times \cdots \times C_k = (A_{\ell})^k$$
and by applying Lemma \ref{l:kov}, \cite[p.328, Lemma]{Scott} and the transitivity of $P \leqs S_k$ we deduce that 
\begin{equation}\label{e:decc}
X \cap {\rm Soc}(Q \wr P) = \prod_{i=1}^{k/t}D_i \cong (A_{\ell})^{k/t}
\end{equation}
is a subdirect product of ${\rm Soc}(Q \wr P)$ for some divisor $t$ of $k$, where each 
\begin{equation}\label{e:didef}
D_i  = \{(z,z^{\a_{i,1}}, \ldots, z^{\a_{i,t-1}}) \mid z \in A_{\ell}\} \cong A_{\ell}
\end{equation}
is a diagonal subgroup of a direct product $\prod_{j \in I_{i}}C_j = (A_{\ell})^{t}$. Here $\a_{i,j} \in {\rm Aut}(A_{\ell}) = S_{\ell}$ for all $i,j$ (recall that $\ell \geqs 7$), and the $I_i$ form a partition of $\Delta = [k]$ into subsets of size $t$. To simplify notation we will assume that $I_{i} = \{(i-1)t+1, \ldots, (i-1)t+t\}$ for all $1 \leqs i \leqs k/t$. Note that 
\begin{equation}\label{e:gbd}
|G| \geqs |T|^{k}|P||A_{\ell}|^{k/t}.
\end{equation}
We now consider two subcases, according to the value of $t$.

\vs

\noindent \textbf{Case 2.1.} $t=1$

\vs

\noindent As in the proof of Proposition \ref{p:pybs}, 
the set $\mathcal{B}$ in \eqref{e:base} is a base for $G$, where $b=b(H)$ satisfies the upper bound in \eqref{e:bb} for some absolute constant $c_3$ (see Theorem \ref{t:d}). Therefore 
$$b(G) \leqs \lceil a/r \rceil+b \leqs \left\lceil c_1\frac{1}{\lfloor \log |\Gamma| \rfloor} + c_2\frac{\log |P|}{k\lfloor \log |\Gamma| \rfloor} \right\rceil+c_3\frac{\log |H|}{\log |\Gamma|}$$
for some absolute constants $c_1,c_2,c_3$. Since $|H| \leqs |T|^2|Q|$ (see \eqref{e:hordd}) it follows that
$$b(G) \leqs c_4\frac{\log |P|}{\log |\Omega|} + c_5\frac{\log(|T|^k|A_{\ell}|^k)}{\log |\Omega|} \leqs c_6 \frac{\log |G|}{\log |\Omega|}$$
as required (the final inequality follows from \eqref{e:gbd}).

\vs

\noindent \textbf{Case 2.2.} $t>1$

\vs

\noindent By replacing $G$ by a suitable conjugate $G^y$ with $y \in \prod_{i=1}^{k}{\rm Sym}(\Gamma_i)<{\rm Sym}(\Omega)$,
we may assume that the automorphisms $\a_{i,j} \in {\rm Aut}(A_{\ell})$ appearing in \eqref{e:didef} are independent of $i$, so that    
\begin{equation}\label{e:dj3}
D_1  = D_2 = \cdots = D_{k/t} = \{(z,z^{\s_1}, \ldots, z^{\s_{t-1}}) \mid z \in A_{\ell}\} \cong A_{\ell}
\end{equation}
for some $\s_j \in {\rm Aut}(A_{\ell}) = S_{\ell}$.

Let $\{\gamma_0, \ldots, \gamma_{b-1}\} \subseteq \Gamma$ be a base for $H$, where $b=b(H)$, and write $b=qt+v$ with $0 \leqs v<t$. Set 
\begin{equation}\label{e:ee}
e=\lceil b/t \rceil.
\end{equation}
We will use the structure of $X \cap {\rm Soc}(Q \wr P)$ 
in \eqref{e:decc} and \eqref{e:dj3} above to construct a specific collection of elements $\{\a_{0}', \ldots, \a_{e-1}'\}$ in $\Omega$. 

For $j \in [k]$ write $j=q't+v'$ with $0 \leqs v'<t$. For $0 \leqs i <e-1$ we define the $j$-th coordinate of $\a_{i}'$ to be the element 
$(\gamma_{it+v'})^{\s_{v'}} \in \Gamma$, where $\s_{0}=1$.
Note that each $\gamma \in \Gamma$ corresponds to a coset of $\ell$-tuples of elements in $S$ by
the diagonal subgroup $D=\{(s,\ldots,s) \mid s \in S\}$ of $S^{\ell}$, so $\gamma^{\s} \in \Gamma$ is defined in
the natural way for all $\s \in S_{\ell}$. Explicitly, if we write $\gamma = (s_1, \ldots, s_{\ell})D$ then
$$\gamma^{\s} = (s_{1^{\s^{-1}}}, \ldots, s_{\ell^{\s^{-1}}})D.$$
For example, the elements $\a_{0}'$ and $\a_{1}'$ in $\Omega$ are defined as follows:
\begin{align*}
\a_{0}' & = \left(\gamma_0, (\gamma_1)^{\s_{1}}, \ldots, (\gamma_{t-1})^{\s_{t-1}}, \gamma_0, (\gamma_{1})^{\s_{1}}, \ldots, (\gamma_{t-1})^{\s_{t-1}}, \ldots\right) \\
\a_{1}' & = \left(\gamma_t, (\gamma_{t+1})^{\s_{1}}, \ldots, (\gamma_{2t-1})^{\s_{t-1}}, \gamma_t, (\gamma_{t+1})^{\s_{1}}, \ldots, (\gamma_{2t-1})^{\s_{t-1}}, \ldots\right)
\end{align*}
We define the $j$-th coordinate of $\a_{e-1}'$ to be 
$(\gamma_{(e-1)t+v'})^{\s_{v'}}$ if $v=0$, otherwise it is $(\gamma_{(e-1)t+v'})^{\s_{v'}}$ if $v' < v$, and 
$(\gamma_{(e-1)t+v-1})^{\s_{v'}}$ if $v' \geqs v$ (note that $(e-1)t+v-1=b-1$ if $v>0$). 

Set
\begin{equation}\label{e:base2}
\mathcal{B}'=\{\a_1, \ldots, \a_{\lceil a/r \rceil}\} \cup \{\a_{1}'',\a_{2}'',\a_{3}''\} \cup \{\a_{0}', \ldots, \a_{e-1}'\},
\end{equation}
where the $\a_i$ are the elements given in Lemma \ref{l:PP}, and the $\a_i''$ are defined in the discussion preceding \eqref{e:aldd}.

\vs

\noindent \textbf{Claim.} \emph{The set $\mathcal{B}'$ in \eqref{e:base2} is a base for $G$.}

\begin{proof}[Proof of claim]
In order to see this, it suffices to show that if $x=(q_1, \ldots, q_k)p \in X$ fixes each $\a_i'$ then $q_j=1$ for all $j$. Indeed, at the start of Case 2 we noted that the pointwise stabilizer in $G$ of the $\a_i''$ is contained in $X$, and Lemma \ref{l:PP} implies that there are no nontrivial elements in $G$ of the form $(1,\ldots, 1)p$ (with $p \in P$) that fix each $\a_i$.

In view of \eqref{e:decc} and \eqref{e:dj3}, we may write
$$x=\left(q_1,q_1^{\s_{1}}, \ldots, q_{1}^{\s_{t-1}}, q_2, q_2^{\s_{1}}, \ldots, q_{2}^{\s_{t-1}}, \ldots, q_{k/t},q_{k/t}^{\s_{1}}, \ldots, q_{k/t}^{\s_{t-1}}\right)p$$
for some $q_j \in S_{\ell}$. Here the element $p \in S_k$ must be compatible with the linking of the $A_{\ell}$ factors in $X \cap {\rm Soc}(Q \wr P)$. More precisely, since $x$ normalizes $X \cap {\rm Soc}(Q \wr P)$, we can write 
$$p^{-1} = (p_0, \ldots, p_{t-1}) \in \prod_{i=1}^{t}{\rm Sym}(\Delta_i)<{\rm Sym}(\Delta)$$
where $\Delta_i = \{i,t+i, \ldots, k-t+i\}$ and $\Delta = [k]$. 

If we view each $p_i \in {\rm Sym}(\Delta_i)$ as a permutation of $[k/t]=\{1, \ldots, k/t\}$ then by definition of the action of $G$ on $\Omega$ (see \eqref{e:prodd}) we have
\begin{align*}
(\a_0')^{x} = & \, \left(\left(\gamma_0\right)^{q_{1^{p_0}}},\left((\gamma_1)^{\s_1}\right)^{q_{1^{p_1}}^{\s_1}}, \ldots, \left((\gamma_{t-1}\right)^{\s_{t-1}})^{q_{1^{p_{t-1}}}^{\s_{t-1}}}, \ldots\right. \\
& \;\;\left. \ldots,  \left(\gamma_0\right)^{q_{(k/t)^{p_0}}},\left((\gamma_1\right)^{\s_1})^{q_{(k/t)^{p_1}}^{\s_1}}, \ldots, \left((\gamma_{t-1}\right)^{\s_{t-1}})^{q_{(k/t)^{p_{t-1}}}^{\s_{t-1}}}\right). 
\end{align*}

Suppose $x$ fixes $\a_{0}'$. By comparing the coordinates of $\a_0'$ and $(\a_0')^x$ corresponding to the points in $\Delta_1$ we deduce that
$$\gamma_0 = (\gamma_0)^{q_{i^{p_0}}}$$
for all $i \in [k/t]$, hence $q_j$ fixes $\gamma_0$ for all $1 \leqs j \leqs k/t$. Similarly, by comparing the coordinates in $\Delta_2$ we see that
$$(\gamma_1)^{\s_1} = \left((\gamma_1)^{\s_1}\right)^{q_{i^{p_1}}^{\s_1}}$$
for all $i \in [k/t]$. Therefore, each $q_j^{\s_1}$ fixes $(\gamma_1)^{\s_1}$ and by applying $\s_1^{-1}$ we deduce that $q_j$ fixes $\gamma_1$.
In general, if $x$ fixes $\a_{0}'$ then 
$$(\gamma_i)^{\s_i} = \left((\gamma_i)^{\s_i}\right)^{q_{j}^{\s_i}}$$
for all $0 \leqs i \leqs t-1$ and all $1 \leqs j \leqs k/t$ (where $\s_0 = 1$ as before), and we conclude that each $q_j$ fixes $\{\gamma_{0}, \ldots, \gamma_{t-1}\}$ pointwise.

By the same argument, if $x$ fixes $\a_{1}'$ then each $q_j$ fixes $\{\gamma_{t}, \ldots, \gamma_{2t-1}\}$ pointwise, and so on. In particular, if $x$ fixes all the $\a_{i}'$ (for $0 \leqs i \leqs e-1$) then each $q_j$ fixes every element in the set $\{\gamma_{0}, \ldots, \gamma_{b-1}\}$, but this is a base for $H$, so only the identity element in $H$ has this property. Therefore $q_j=1$ for all $j$. We conclude that the set 
$\mathcal{B}'$ in \eqref{e:base2} is a base for $G$.
\end{proof}

Let us now complete the proof of Proposition \ref{p:pybd}. Recall that $e=\lceil b/t \rceil$ (see \eqref{e:ee}), and $b=b(H)$ satisfies the upper bound in \eqref{e:bb}. In view of \eqref{e:aa}, \eqref{e:rdef} and \eqref{e:hordd} we deduce that there are absolute constants $c_i$ such that
\begin{align*}
b(G) \leqs \lceil a/r \rceil +\lceil b/t \rceil + 3 & \leqs \left\lceil c_1\frac{1}{\lfloor \log |\Gamma| \rfloor} + c_2\frac{\log |P|}{k\lfloor \log |\Gamma| \rfloor} \right\rceil + c_3\frac{\log |H|}{t\log |\Gamma|} \\
& \leqs c_4\frac{\log |P|}{\log |\Omega|}+c_5\frac{k\log (|S|^{\ell}|A_{\ell}|)}{t\log |\Omega|} \\
& \leqs c_6\frac{\log |G|}{\log |\Omega|}
\end{align*}
as required (the final inequality follows from \eqref{e:gbd}).
\end{proof}

\vs

By combining Propositions \ref{p:pybs} and \ref{p:pybd}, this completes the proof of Theorem \ref{t:pybppa}.  
\end{proof}

\section{Twisted wreath products}\label{s:tw}

In this final section we prove Pyber's conjecture for twisted wreath products as an easy corollary of Theorem \ref{t:pybppa} on product-type groups. 

\begin{thm}\label{t:tw}
Pyber's conjecture holds if $G$ is a twisted wreath product.
\end{thm}

\begin{proof}
Let $G \leqs {\rm Sym}(\Omega)$ be a primitive twisted wreath product group with socle $T^k$, where $T$ is a  nonabelian simple group and $k \geqs 6$. Note that $T^k$ is a regular normal subgroup of $G$. Let $P=G_{\a}$ be the stabilizer of a point $\a \in \Omega$. Then $G=T^kP$ is a semidirect product, and the top group $P$ is a transitive subgroup of $S_k$. As explained in \cite[Section 3.6]{Praeger}, we can embed $G$ in a primitive group $L \leqs {\rm Sym}(\Omega)$ of product-type, where $L=T^2 \wr P = (T^2)^k.P$ (in terms of the notation in Table \ref{t:onan}, $L = H \wr P$ is of type III(b)(ii) and $H$ is of type III(a)(ii)). By Theorem \ref{t:pybppa}, there exists an absolute constant $c$ such that
$$b(L) \leqs c \frac{\log |L|}{\log |\Omega|},$$
whence
$$b(G) \leqs b(L) \leqs c \frac{\log |L|}{\log |\Omega|} < 2c \frac{\log |G|}{\log |\Omega|}.$$
The result follows. 
\end{proof}

\vs

In view of Theorems \ref{t:as}, \ref{t:d}, \ref{t:pybppa} and \ref{t:tw}, the proof of Theorem \ref{tt:main} is complete.


\begin{thebibliography}{99}

\bibitem{Benbenishty}
C. Benbenishty, \emph{On actions of primitive groups}, PhD thesis, Hebrew University, Jerusalem, 2005.

\bibitem{BCN}
C. Benbenishty, J.A. Cohen and A.C. Niemeyer, \emph{The minimum length of a base for the symmetric group acting on partitions}, European J. Comb. \textbf{28} (2007), 1575--1581.

\bibitem{Bur} T.C. Burness, \emph{On base sizes for actions of finite classical groups}, J. London Math. Soc. \textbf{75} (2007), 545--562.

\bibitem{BGS} T.C. Burness, R.M. Guralnick and J. Saxl, \emph{On base sizes for symmetric groups}, Bull. London Math. Soc. \textbf{43} (2011), 386--391.

\bibitem{BLS} T.C. Burness, M.W. Liebeck and A. Shalev, \emph{Base sizes for simple groups and a conjecture of Cameron}, Proc. London Math. Soc. \textbf{98} (2009), 116--162.

\bibitem{BOW} T.C. Burness, E.A. O'Brien and R.A. Wilson, \emph{Base sizes for sporadic simple groups}, Israel J. Math. \textbf{177} (2010), 307--334.

\bibitem{Dolfi} 
S. Dolfi, \emph{Orbits of permutation groups on the power set}, Arch. Math. (Basel) \textbf{75} (2000), 321--327.


\bibitem{Fawcett}
J.B. Fawcett, \emph{The base size of a primitive diagonal group}, J. Algebra \textbf{375} (2013), 302--321.

\bibitem{GM} D. Gluck and K. Magaard, \emph{Base sizes and regular orbits for coprime affine permutation groups}, J. London Math. Soc. \textbf{58} (1998), 603--618.

\bibitem{GSS} D. Gluck, {\'A}. Seress and A. Shalev, \emph{Bases for primitive permutation groups and a conjecture of {B}abai}, J. Algebra \textbf{199} (1998), 367--378.

\bibitem{Halasi} Z. Halasi, \emph{On the base size for the symmetric group acting on subsets}, Studia Sci. Math. Hungar. \textbf{49} (2012), 492--500.

\bibitem{HP} Z. Halasi and K. Podoski, \emph{Every comprime linear group admits a base of size two}, preprint (arXiv:1212.0199v1)

\bibitem{James} J.P. James, \emph{Partition actions of symmetric groups and regular bipartite graphs}, Bull. London Math. Soc. \textbf{38} (2006), 224--232.

\bibitem{Kovacs} L.G. Kov\'{a}cs, \emph{Primitive subgroups of wreath products in product action}, Proc. London Math. Soc. \textbf{58} (1989), 306--322.

\bibitem{LPS}
M.W. Liebeck, C.E. Praeger and J. Saxl, \emph{On the O'Nan-Scott theorem for finite primitive permutation groups}, J. Austral. Math. Soc. \textbf{44} (1988), 389--396.

\bibitem{LS0} M.W. Liebeck and A. Shalev, \emph{Simple groups, permutation groups, and probability}, J. Amer. Math. Soc. \textbf{12} (1999), 497--520.


\bibitem{LS} M.W. Liebeck and A. Shalev, \emph{Bases of primitive linear groups}, J. Algebra \textbf{252} (2002), 95--113.

\bibitem{Praeger} C.E. Praeger, \emph{The inclusion problem for finite primitive permutation groups}, Proc. London Math. Soc. \textbf{60} (1990), 68--88.

\bibitem{Pyber}
L. Pyber, \emph{Asymptotic results for permutation groups}, in Groups and Computation (eds. L. Finkelstein and W. Kantor), DIMACS Series, Vol. 11, pp. 197--219, 1993.

\bibitem{Scott}
L.L. Scott, \emph{Representations in characteristic $p$}, The Santa Cruz
  Conference on Finite Groups, Proc. Sympos. Pure Math., Vol. 37,  pp. 319--331, 1980. 
  
\bibitem{Seress} {\'A}. Seress, \emph{The minimal base size of primitive solvable permutation groups}, J. London Math. Soc. \textbf{53} (1996), 243--255.

\bibitem{Akos1} {\'A}. Seress, \emph{Primitive groups with no regular orbits on the set of subsets}, Bull. London Math. Soc. \textbf{29} (1997), 697--704.

\bibitem{Seress_book}
{\'A}. Seress, \emph{Permutation {G}roup {A}lgorithms}, Cambridge
  Tracts in Mathematics, vol. 152, Cambridge University Press, Cambridge, 2003.
  
\end{thebibliography}
\end{document}